\documentclass{amsart}
\usepackage{amsmath,amssymb,amsthm,amsfonts,amscd,mathrsfs,mathtools,enumerate}
\usepackage{tikz,tikz-cd,caption}
\usepackage[all]{xy}
\usepackage[hyphens]{url}
\usepackage{hyperref}
\usepackage{caption}
\usetikzlibrary{quotes,arrows}

\theoremstyle{plain}
    \newtheorem{thm}{Theorem}[section]
    \newtheorem{lem}[thm]   {Lemma}
    
    \newtheorem{prop}[thm]  {Proposition}
    
\theoremstyle{definition}
    \newtheorem{defn}[thm]  {Definition}

    \newtheorem{rem}[thm]{Remark}

\def\A{{\mathcal A}}

\def\cat{\mathsf{cat}}
\def\secat{\mathsf{secat}}
\def\wgt{\mathrm{wgt}}

\def\dim{\mathrm{dim}}
\def\ker{\operatorname{ker}}

\def\coker{\mathrm{coker}}

\newcommand{\be}{\begin{enumerate}}
\newcommand{\ee}{\end{enumerate}}

\newcommand{\Z}{\mathbb{Z}}

\newcommand{\C}{\mathbb{C}}
\newcommand{\Q}{\mathbb{Q}}
\newcommand{\Ha}{\mathbb{H}}
\newcommand{\F}{\mathbb{F}}

\newcommand{\TC}{{\sf TC}}

\def\Mwgt{\mathrm{Mwgt}}
\def\Swgt{\mathrm{Swgt}}

\begin{document}

\title[Secondary operations]{Secondary cohomology operations and sectional category}

\author{Mark Grant}

\address{Institute of Mathematics,
Fraser Noble Building,
University of Aberdeen,
Aberdeen AB24 3UE,
UK}

\email{mark.grant@abdn.ac.uk}

\date{\today}


\begin{abstract} 
We show how secondary cohomology operations in the total space of the fibred join can be used to give lower bounds for the sectional category of a fibration. This suggests a refinement of the module weight of Iwase--Kono, which we call the {\em secondary module weight}. Examples are given for which the secondary module weight at the prime $2$ detects sectional category while the module weight does not.
\end{abstract}

\maketitle
\section{Introduction}\label{sec:intro}

The {\em sectional category} of a fibration $q:E\to B$, denoted $\secat(q)$, is defined to be the minimal integer $k$ such that $B$ admits a cover by open sets $U_0,\ldots , U_k$, on each of which $q$ admits a continuous local section $s_i:U_i\to E$ \cite{Schwarz,James}. Prominent special cases include the Lusternik--Schnirelmann category $\cat(X)$ \cite{CLOT}, which is the sectional category of the based path fibration $P_0 X\to X$, and Farber's topological complexity $\TC(X)$ \cite{Far03,Far04}, which is the sectional category of the free path fibration $PX\to X\times X$, as well as its many variants \cite{FG07,Rudyak,CFW}. Note that here we use the normalized version, for which $\secat(q)=0$ if and only if $q$ admits a (global) section. 

An important foundational result due to Schwarz \cite{Schwarz} (see also \cite{James}) states that if $B$ is paracompact then $\secat(q)$ agrees with the minimal $k$ such that the $(k+1)$-fold fibred join $q(k):E(k)\to B$ admits a section. If $F$ is the fibre of $q$, then the fibre of $q(k)$ is $F(k)=F\ast \cdots \ast F$, the iterated join of $(k+1)$ copies of $F$. When $F$ is $(r-1)$-connected, $F(k)$ is $(r+1)(k+1)-2$-connected. Then (assuming $B$ is a CW complex) cellular obstruction theory gives an upper bound
\[
\secat(q)< \frac{\dim(B)+1}{r+1}.
\]

A standard cohomological argument using the open sets definition gives that $\secat(q)$ is bounded below by $\operatorname{nil}\ker (q^*:H^*(B;R)\to H^*(E;R))$, the nilpotency of the kernel of the induced map in cohomology. Here we take coefficients in an arbitrary commutative ring $R$ (local coefficients may also be used). This bound can sometimes be improved upon using the fibred join interpretation. In particular, we define following Farber--Grant \cite{FG07,FG08} the \emph{sectional category weight} of $q$ to be 
\[
\wgt(q;R)=\min \{k \mid q(k)^*:H^*(B;R)\to H^*(E(k);R)\mbox{ is injective}\}.
\]
Then (since a section $s:B\to E(k)$ of $q(k)$ clearly implies injectivity of $q(k)^*$) one has $\secat(q)\geq\wgt(q;R)$. It can be shown that $\wgt(q;R)\geq \operatorname{nil}\ker (q^*:H^*(B;R)\to H^*(E;R))$, and examples of strict inequality are known.

A further improvement to this lower bound was found by Iwase and Kono \cite{IwaseKono}. They observed that a section $s$ of $q(k)$, in addition to implying injectivity of the induced map $q(k)^*$, would entail a retraction $s^*: H^*(E(k);R)\to H^*(B;R)$ which commutes with all cohomology operations. They therefore defined the \emph{module weight} of $q$ at the prime $p\ge2$ to be
\begin{multline*}
\Mwgt(q;\F_p)=\min \{k \mid q(k)^*:H^*(B;\F_p)\to H^*(E(k);\F_p)\mbox{ admits}\\
\mbox{a retraction in the category of $\mathcal{A}_p$-modules}\}.
\end{multline*}  
Here $\mathcal{A}_p$ denotes the mod $p$ Steenrod algebra. Clearly, one has inequalities
\[
\secat(q)\ge \Mwgt(q;\F_p) \ge \wgt(q;\F_p)\ge \operatorname{nil}\ker(q^*: H^*(B;\F_p)\to H^*(E;\F_p))
\]
for each prime $p$. Examples of strict inequality $\Mwgt(q;\F_p) > \wgt(q;\F_p)$ can be found in \cite{IwaseKono, Choi09, Choi15}.

The ideas in the previous two paragraphs were originally developed in the study of Lusternik--Schnirelmann category, where $q$ is the based path fibration $P_0 X\to X$. In this context $\wgt(q;R)$ coincides with the \emph{strict category weight} of $X$, as developed by Rudyak \cite{Rudyak99} and Strom \cite{Strom} based on ideas of Fadell and Husseini \cite{FH}. If in addition $R$ is a field, then $\wgt(q;R)$ coincides with the Toomer invariant \cite{Toomer}. With $R=\Q$ the rationals, the invariant $\wgt(q;\Q)$ for arbitrary fibrations $q: E\to B$ has been introduced in \cite{C-VKV} where it is denoted $\operatorname{Hsecat}(q)$. 

In this paper we will improve on the module weight using secondary cohomology operations. Recall that whenever stable cohomology operations $\theta$ and $\varphi$ satisfy a relation $\varphi\circ\theta\equiv 0$, this gives rise to a \emph{secondary cohomology operation}
\[
\Phi=\Phi(\varphi,\theta) : \ker \,\theta \to \coker \,\varphi
\]
with \emph{indeterminacy} $\operatorname{im} \varphi$. Such secondary operations are natural, in the following sense: if $f: X\to Y$ is a map of spaces, then the diagram
 \[
 \xymatrix{
 \ker\, \theta_Y \ar[r]^-{\Phi} \ar[d]_{f^*} & \coker\, \varphi_Y \ar[d]^{f^*} \\
 \ker\, \theta_X \ar[r]^-{\Phi} & \coker\, \varphi_X
 }
 \]
 commutes. (This makes sense because $f^* (\ker\,\theta_Y)\subseteq \ker\,\theta_X$ and $f^*(\operatorname{im} \varphi_Y)\subseteq \operatorname{im} \varphi_X$.) If we consider only stable cohomology operations with $\F_p$-coefficients, we can say that $f^*:H^*(Y;\F_p)\to H^*(X;\F_p)$ is a map of $\A_p$-modules which commutes with all secondary operations.  This leads us to our main definition.

\begin{defn}\label{def:Swgt}
The \emph{secondary module weight} of the fibration $q:E\to B$ at the prime $p$ is 
 \begin{multline*}
 \Swgt(q;\F_p) := \min \{k \mid q(k)^*:H^*(B;\F_p)\to H^*(E(k);\F_p)\mbox{ admits a retraction } \\ 
 \mbox{in the category of $\mathcal{A}_p$-modules that commutes with all secondary operations}\}.
\end{multline*}
\end{defn}

We illustrate the utility of this definition by giving examples where $\secat(q)=\Swgt(q;\F_p)>\Mwgt(q;\F_p)$. Our first example is the so-called \emph{twistor bundle} 
\[
\xymatrix{
S^2 \ar[r] & \C P^5 \ar[r]^q & \Ha P^2
}
\]
under which $q$ sends a complex line in $\C^6\cong\Ha^3$ to the quaternionic line in $\Ha^3$ containing it. 
Standard techniques give $1\le \secat(q)\le 2$, but are inconclusive in determining the exact value of $\secat(q)$. In fact, it was our attempts to determine whether $q(1)$ admits a section using obstruction theory that led us to Definition \ref{def:Swgt}.

\begin{thm}\label{thm:twistor}
The twistor bundle $q:\C P^5\to \Ha P^2$ has 
\[
2= \secat(q)=\Swgt(q;\F_2)>\Mwgt(q;\F_2)=1>\wgt(q;\F_2)=0.
\]
\end{thm}
Theorem \ref{thm:twistor} answers a question on MathOverflow asked by Prateep Chakraborty \cite{Chakra}. Note that the rational sectional category $\secat_0(q)$ in this example is $1$, as follows from results of Stanley \cite{Stanley}.

Our second example concerns Lusternik--Schnirelmann category. As is standard, when $q:P_0 X\to X$ is the based path fibration on a space $X$ we denote $\wgt(q;\F_p)$, $\Mwgt(q;\F_p)$ and $\Swgt(q;\F_p)$ by $\wgt(X;\F_p)$, $\Mwgt(X;\F_p)$ and $\Swgt(X;\F_p)$, respectively.

\begin{thm}\label{thm:2cell}
Let $\eta\in \pi_7(S^6)$ and $\omega\in \pi_6(S^3)$ be generators, and let $\alpha\in\pi_7(S^3)$ be the homotopy class of the composition $\omega\circ\eta$. If $X=S^3\cup_\alpha e^8$ is the homotopy cofibre of $\alpha$, then
\[
2=\cat(X)=\Swgt(X;\F_2)>\Mwgt(X;\F_2)=\wgt(X;\F_2)=1.
\]
\end{thm}

In the proofs of both Theorems \ref{thm:twistor} and \ref{thm:2cell}, the secondary operation $\Phi$ based on the Adem relation $Sq^3Sq^1+Sq^2Sq^2\equiv 0$ is shown to be nonzero with zero indeterminacy in the total space of the fibred join. In each case, this operation detects the Hopf invariant of the attaching map of the top cell \cite{BH,CLOT,GGV}, which turns out to be represented by the class of the composition $\eta\circ\eta\in \pi_7(S^5)$. The main technique to prove non-triviality of $\Phi$ is what Harper calls ``compatibility with exact sequences" in his book \cite{Harper}. This method sometimes allows to compute secondary operations in a space $Y$ in terms of primary operations in the other spaces in a cofiber sequence involving $Y$. We review the method in Section 2 below. Thereafter the proofs diverge: Theorem \ref{thm:twistor} uses the Thom isomorphism relating Steenrod squares in the total space of a sphere bundle with Stiefel-Whitney classes, while Theorem \ref{thm:2cell} uses results of Kono--Kozima \cite{KonoKozima} on Steenrod squares in the loop space of $\operatorname{Spin}(n)$. 

The author wishes to thank John Oprea and Lucile Vandembroucq for helpful comments on an earlier draft, and Prateep Chakraborty for asking the question \cite{Chakra} which initiated this work. 

\section{Secondary cohomology operations}

In this section we recall the neccessary material about secondary cohomology operations, following Harper's book \cite[Chapter 4]{Harper}. Our main interest will be in cohomology operations with $\F_2$-coefficients; however, since we will need to consider inhomogeneous primary operations we will state things in greater generality than we perhaps need. 

We work in the category of compactly generated spaces with non-degenerate base points. Consider a composition of maps
\begin{equation}\label{eq:compos}
\xymatrix{
K_0 \ar[r]^-{\theta} & K_1 \ar[r]^-{\varphi} & K_2
}
\end{equation}
such that $\varphi\circ\theta\sim \ast$ and $K_2$ is a simply-connected $H$-space. Given a space $X$ we write
\[
\ker \theta_X:=\{ [\varepsilon] \mid \varepsilon: X\to K_0,\, \theta\circ\varepsilon \sim \ast\} \subseteq [X, K_0],
\]
and 
\[
\coker\, \Omega\varphi_X:=[X,\Omega K_2]/\operatorname{im} \Omega\varphi_*,
\]
where $\Omega\varphi_* : [X,\Omega K_1] \to [X,\Omega K_2]$ is given by $\Omega\varphi_*(g)=\Omega\varphi\circ g$. (Harper writes $S_\theta(X)$ for $\ker \theta_X$ and $T_{\Omega\varphi}(X)$ for $\coker\,\Omega\varphi_X$.) These constructions are functorial: if $f:X\to Y$ is a map of spaces, there are induced maps
\[
f^*: \ker\theta_Y \to \ker\theta_X,\qquad f^*[\varepsilon]=[\varepsilon\circ f],
\]
\[
f^*: \coker\,\Omega\varphi_Y\to \coker\,\Omega\varphi_X,\qquad f^*[[g]]=[[g\circ f]].
\]

Given a composition as in (\ref{eq:compos}) together with a null-homotopy $H: \varphi\circ\theta \sim \ast$, there is defined a \emph{secondary cohomology operation} based on the triple $(\varphi,\theta,H)$, which is a natural transformation
\[
\Phi=\Phi(\varphi,\theta,H): \ker \theta_{(-)}\to \coker\,\Omega\varphi_{(-)}.
\]
 
In most applications the spaces $K_0$, $K_1$ and $K_2$ are products of Eilenberg--MacLane spaces, and $\varphi$ and $\theta$ represent combinations of Steenrod operations. For example, for each $i\ge0$ there are maps $\theta: K(\F_2,i)\to K(\F_2,i+1)\times K(\F_2,i+2)$ and $\varphi: K(\F_2,i+1)\times K(\F_2,i+2)\to K(\F_2,i+4)$ representing the inhomogenous primary operations 
\[
\begin{pmatrix} Sq^1 & Sq^2 \end{pmatrix}: H^i(X)\to H^{i+1}(X)\times H^{i+2}(X),\quad x\mapsto \big(Sq^1(x), Sq^2(x)\big),
\]
\[
\begin{pmatrix} Sq^3 \\ Sq^2 \end{pmatrix}: H^{i+1}(X)\times H^{i+2}(X) \to H^{i+4}(X), \quad (x,y) \mapsto Sq^3(x) + Sq^2(y).
\]
(Here $\F_2$-coefficients are understood.) That $\varphi\circ\theta\sim \ast$ is null-homotopic is witnessed by the Adem relation $Sq^3Sq^1 + Sq^2Sq^2\equiv 0$. There results a homogeneous secondary operation $\Phi: H^i(X)\dashrightarrow H^{i+3}(X)$, defined on the kernel of $\begin{pmatrix} Sq^1 & Sq^2 \end{pmatrix}$ and taking values in the cokernel of $\begin{pmatrix} Sq^3 \\ Sq^2 \end{pmatrix}$. In this example $\Phi$ is independent of the choice of null-homotopy $H:\varphi\circ\theta\sim\ast$, although not in general \cite[p.\ 89]{Harper}.

In the above example we in fact have sequences of composable maps $\theta=\{\theta_i\}_{i\ge0}$ and $\varphi=\{\varphi_i\}_{i\ge0}$ which represent stable cohomology operations in the sense that $\Omega \theta_{i+1}=\theta_i$ and $\Omega\varphi_{i+1}=\varphi_i$ for all $i\ge0$. We therefore have that $\ker\theta_X\subseteq H^*(X;\F_p)$ and $\operatorname{im} \Omega\varphi_X =\operatorname{im} \varphi_X\subseteq H^*(X;\F_p)$ are graded vector subspaces. We will say that a graded linear map $F: H^*(Y;\F_p)\to H^*(X;\F_p)$ of degree $0$ \emph{commutes with all secondary operations} if for every composable pair of stable operations $\theta=\{\theta_i\}_{i\ge0}$ and $\varphi=\{\varphi_i\}_{i\ge0}$ as above with $\varphi\circ\theta\sim \ast$, one has $F (\ker\theta_Y)\subseteq \ker\theta_X$ and $F(\operatorname{im}\varphi_Y) \subseteq \operatorname{im}\varphi_X$, and the diagram
   \[
 \xymatrix{
 \ker\, \theta_Y \ar[r]^-{\Phi} \ar[d]_{F} & \coker\, \varphi_Y \ar[d]^{F} \\
 \ker\, \theta_X \ar[r]^-{\Phi} & \coker\, \varphi_X
 }
 \] 
commutes. In particular, since $\Phi$ is a natural transformation any $F=f^*$ induced by a map of spaces $f:X\to Y$ commutes with all secondary operations.

There are many tools for computing secondary operations. We will only make use of what Harper calls ``compatibility with exact sequences" \cite[Section 4.2.3]{Harper}. Let $f: X\to Y$ be a map and consider the cofibration sequence
\begin{equation}\label{eq:cofib}
\xymatrix{
 X \ar[r]^-{f} & Y \ar[r]^-{j} & C_f \ar[r]^-{\tau} & \Sigma X \ar[r]^-{\Sigma f} & \Sigma Y
 }
 \end{equation}
 where $C_f$ is the homotopy cofibre of $f$, and $j$ and $\tau$ include and collapse $Y$, respectively. Let $\Phi: \ker\theta_{(-)}\to \coker\,\Omega\varphi_{(-)}$ be the secondary operation based on a null-homotopic composition
 \[
 \xymatrix{
 K_0 \ar[r]^-{\theta} & K_1 \ar[r]^-{\varphi} & K_2,
 }
 \]
 where we assume in addition that $K_1$ is an $H$-space and $\varphi$ is an $H$-map. The result below gives a means of computing $\Phi: \ker\theta_Y\to \coker\,\Omega\varphi_Y$ in terms of the primary operations represented by $\theta$ and $\varphi$ in the other spaces in (\ref{eq:cofib}), and is independent of the choice of null-homotopy of $\varphi\circ\theta$.
 
We have a commutative diagram in which the middle column is exact:
\[
\xymatrix{
  & [\Sigma Y, K_1] \ar@{=}[r] \ar[d]_{\Sigma f^*} & [\Sigma Y, K_1] \ar[d]_B \\
   & [\Sigma X,K_1] \ar[r]^{\varphi_*} \ar[d]_{\tau^*} & [\Sigma X, K_2]\\
 [C_f, K_0]\ar[r]^-{\theta_*} \ar@{=}[d] &  [C_f, K_1] \ar[d]_{j^*} & \\
 [C_f, K_0] \ar[r]^-{A} & [Y, K_1]
 }
 \]
 Here the operations $A: [C_f, K_0] \to [Y, K_1]$ and $B: [\Sigma Y, K_1] \to [\Sigma X, K_2]$ are defined by commutativity of the diagram. By exactness of the middle column we can define an operation
 \[
 \Delta=\Delta(\varphi,\theta) : \ker A \to \coker\, B,\qquad \Delta=\varphi_* \circ (\tau^*)^{-1}\circ \theta_*.
 \]
Note that $j^*: [C_f,K_0]\to [Y,K_0]$ maps $\ker A$ to $\ker\theta_Y$, and $(\Sigma f)^*: [\Sigma Y,K_2]\to [\Sigma X,K_2]$ maps $\coker\,\varphi_{\Sigma Y}\cong \coker\,\Omega\varphi_Y$ to $\coker\, B$.
   
\begin{prop}[{\cite[Prop 4.2.2]{Harper}}]\label{prop:compat}
The following diagram commutes:
\[
\xymatrix{
\ker A \ar[r]^-{\Delta} \ar[d]_{j^*} & \coker B \\
\ker\theta_Y \ar[r]_-{\Phi} & \coker\,\Omega\varphi_Y \cong \coker\,\varphi_{\Sigma Y} \ar[u]_-{-(\Sigma f)^*}
}
\]
\end{prop}

Our particular example of interest is the operation $\Phi:H^i(X)\dashrightarrow H^{i+3}(X)$ based on the Adem relation $Sq^3Sq^1 + Sq^2Sq^2\equiv 0$. Then a cofibration sequence as in (\ref{eq:cofib}) induces a diagram
\[
\xymatrix{
  && H^{i+1}(\Sigma Y)\times H^{i+2}(\Sigma Y) \ar@{=}[r] \ar[d]_{\Sigma f^*} & H^{i+1}(\Sigma Y)\times H^{i+2}(\Sigma Y) \ar[d]_B \\
   && H^{i+1}(\Sigma X)\times H^{i+2}(\Sigma X) \ar[r]^-{\begin{pmatrix}Sq^3 \\Sq^2\end{pmatrix}} \ar[d]_{\tau^*} & H^{i+4}(\Sigma X)\\
 H^i(C_f) \ar[rr]^-{\begin{pmatrix} Sq^1 & Sq^2 \end{pmatrix}} \ar@{=}[d] &&  H^{i+1}(C_f)\times H^{i+2}(C_f) \ar[d]_{j^*} & \\
 H^i(C_f) \ar[rr]^-{A} && H^{i+1}(Y)\times H^{i+2}(Y)
 }
 \]
 and the operation $\Delta: \ker A \to \coker\, B$ fits into a commutative diagram
 \[
\xymatrix{
\ker A \ar[r]^-{\Delta} \ar[d]_{j^*} & \coker B \\
H^i(Y) \ar@{-->}[r]^-{\Phi}&  H^{i+3}(Y)\cong H^{i+4}(\Sigma Y) \ar[u]_-{(\Sigma f)^*}
}
\]
where the sign on the right vertical map has been dropped since we are working mod $2$.

\section{Proof of Theorem \ref{thm:twistor}} 

Our goal is to compute the sectional category of the twistor bundle
\[
\xymatrix{
S^2 \ar[r] & \C P^5 \ar[r]^q & \Ha P^2.
}
\]
The upper bound from the Introduction immediately gives 
\[
\secat(q)< \frac{8+1}{2+1}=3,\quad\mbox{or }\secat(q)\le2.
\]
For any coefficient ring $R$, the Leray--Serre spectral sequence $$H^p(\Ha P^2; H^q(S^2;R))\implies H^*(\C P^5;R)$$ is concentrated in even degrees, hence collapses at the $E_2$-page. In particular, the map $q^*: H^*(\Ha P^2;R)\to H^*(\C P^5;R)$ is injective, which shows that $\wgt(q;R)=0$.

We can see that $q$ does not have a section, so that $\secat(q)>0$, by looking at homotopy groups. The long exact homotopy sequences of the fibrations
\[
\xymatrix{
S^1 \ar[r] & S^{11} \ar[r] &\C P^5,
}
\qquad 
\xymatrix{
S^3 \ar[r] & S^{11} \ar[r] &\Ha P^2,
}
\]
imply that $\pi_4(\C P^5)=0$, while $\pi_4(\Ha P^2)\cong \pi_3(S^3)\cong\Z$. Thus there cannot be a section $s$ of $q$, since the induced map $\pi_4(s)$ would be a split injection from $\Z$ to $0$.

Thus the question is whether $\secat(q)$ is $1$ or $2$, or equivalently, whether the $2$-fold fibred join
\[
\xymatrix{
S^5=S^2\ast S^2 \ar[r] & E(1) \ar[r]^{q(1)} & \Ha P^2
}
\]
admits a section. Note that the Leray--Serre spectral sequence $$H^p(\Ha P^2; H^q(S^5;R))\implies H^*(E(1);R)$$ for $q(1)$ again collapses for degree reasons, implying that $q(1)^*:H^*(\Ha P^2; R)\to H^*(E(1); R)$ is injective.

\begin{lem}\label{lem:Mwgttwistor}
For $q$ the twistor bundle, $\Mwgt(q;\F_2)=1$.
\end{lem}

\begin{proof}
Let $H^*(-):=H^*(-;\F_2)$. We have
  \[
 H^*(\Ha P^2)\cong \F_2[a]/(a^3),\qquad |a|=4,
 \]
 \[
 H^*(\C P^5) \cong \F_2[b]/(b^6),\qquad |b|=2,
 \]
 and $q^*:H^*(\Ha P^2)\to H^*(\C P^5)$ is given by $q^*(a)=b^2=Sq^2(b)$.
 
 If $q^*$ were to admit a retraction $s^*:H^*(\C P^5)\to H^*(\Ha P^2)$ as $\mathcal{A}_2$-modules, we would have
 \[
 a=s^* q^*(a)=s^*(Sq^2(b))=Sq^2(s^*(b))=Sq^2(0)=0,
 \]
 a contradiction. So $\Mwgt(q;\F_2)\ge 1$. 
 
 On the other hand, collapse of the spectral sequence for $q(1)$ at the $E_2$-page gives
 \begin{equation}\label{eq:cohE1}
 H^*(E(1)) \cong H^*(\Ha P^2)\otimes_{\F_2} H^*(S^5) \cong \begin{cases} \F_2 & \mbox{if }*=0,4,5,8,9,13, \\ 0 & \mbox{else.} \end{cases}
 \end{equation}
The image of $q(1)^*$ is identified with $H^*(\Ha P^2)\otimes H^0(S^5)$. There are no non-trivial Steenrod operations from $H^*(\Ha P^2)\otimes H^5(S^5)$ to $H^*(\Ha P^2)\otimes H^0(S^5)$ (note that $Sq^3=Sq^1Sq^2$). We can therefore define an $\mathcal{A}_2$-module retraction $s^*$ of $q(1)^*$ by setting $s^*(q(1)^*(x))=x$ for $x\in H^*(\Ha P^2)$ and declaring $s^*$ to be zero outside the image of $q(1)^*$.  Hence $\Mwgt(q;\F_2)\le1$, and the lemma is proved.
 \end{proof}
 
 \begin{rem}
 Similar arguments show that $\Mwgt(q;\F_p)=0$ for $p\ge 3$. On the other hand it is clear that $q^*: H^*(\Ha P^2;R)\to H^*(\C P^5;R)$ cannot admit a \emph{multiplicative} retraction, giving another proof that $\secat(q)\ge1$.
 \end{rem}
 
 \begin{lem}\label{lem:Swgttwistor}
 For $q$ the twistor bundle, $\Swgt(q;\F_2)=2$.
 \end{lem}
 
 \begin{proof}
 Again, let $H^*(-):=H^*(-;\F_2)$. Consider the secondary cohomology operation $\Phi$ based on the Adem relation $Sq^3Sq^1 + Sq^2 Sq^2\equiv 0$. It is defined on the intersection of the kernels of $Sq^1$ and $Sq^2$, and has indeterminacy spanned by the images of $Sq^3$ and $Sq^2$. Let $y\in H^5(E(1))$ be the generator, and note that $Sq^1(y)$ and $Sq^2(y)$ land in zero groups by Equation (\ref{eq:cohE1}). Note also that $Sq^3=Sq^1Sq^2:H^5(E(1))\to H^8(E(1))$ and $Sq^2:H^6(E(1))\to H^8(E(1))$ both factor through zero groups, hence are zero. Thus $\Phi(y)$ is a well-defined element of $H^8(E(1))\cong \F_2$.

We claim that $\Phi(y)$ is nonzero, hence equal to $q(1)^*(a^2)$ where $a^2\in H^8(\Ha P^2)$ is the generator. This implies that $\Swgt(q;\F_2)>1$, as follows. Observe that in $\Ha P^2$ the operation $\Phi: H^5(\Ha P^2)\to H^8(\Ha P^2)$ is zero with zero indeterminacy. Suppose there was a retraction $s^*: H^*(E(1))\to H^*(\Ha P^2)$ of $q(1)^*$ which commutes with all secondary operations. Then we would have
\[
a^2 = s^*q(1)^*(a^2) = s^*\Phi(y) = \Phi(s^*y) = 0,
\]
a contradiction.

To prove the claim we use Proposition \ref{prop:compat} to compute $\Phi(y)$ in terms of primary operations, by placing $E(1)$ in a cofibre sequence. To this end, let us recall that for a fibration $q:E\to B$ the total space $E(1)$ of the fibred join is defined by the homotopy pushout square
\[
\xymatrix{
E\times_B E \ar[r]^-{\pi_2} \ar[d]_{\pi_1} & E \ar[d]^f \\
E \ar[r] & E(1)
}
\]
where the maps $\pi_i:  E\times_B E\to E$ are the projections from the pullback of $q$ by itself onto the two factors. It follows that the homotopy cofibre $C$ of the natural map $f:E\to E(1)$ (which includes $E$ as one of the fibred join factors) is equivalent to the homotopy cofibre of the projection $\pi_1: E\times_B E\to E$. 

In our setting $q:\C P^5\to \Ha P^2$ is a sphere bundle with fibre $S^2$, and so the same is true of its pullback $\pi_1: \C P^5 \times_{\Ha P^2} \C P^5\to \C P^5$. We can therefore think of $C$ as the Thom space of the sphere bundle $\pi_1$, in the following sense. The mapping cylinder $M\pi_1$ is thought of as the disk bundle, and $S:=\C P^5 \times_{\Ha P^2} \C P^5$ (the sphere bundle) includes at one end of the cylinder such that $C=M\pi_1/S$. The Thom isomorphism for sphere bundles \cite{Thom} then gives an isomorphism in cohomology
\[
\varphi: H^i(\C P^5) \to \tilde{H}^{i+3}(C)\cong H^{i+3}(M\pi_1,S),\qquad x\mapsto \pi_1^*(x)\cup U.
\]
 Here $U\in \tilde{H}^3(C)\cong H^3(M\pi_1,S)$ is the Thom class. Thom also defined Stiefel--Whitney classes for sphere bundles using the formula
 \[
 \varphi(w_i)=Sq^i U.
 \]
 These are natural with respect to maps of sphere bundles. In particular, for $i=1,2(,3)$ we have 
 \[
 w_i(\pi_1)=q^*(w_i(q))=0,\qquad\mbox{since }H^i(\Ha P^2)=0.
 \]
 It follows that $Sq^1(U)=0$ and $Sq^2(U)=0$ for the Thom class $U\in \tilde{H}^3(C)$.
 
 Now let $H^*(\C P^5)\cong \F_2[b]/(b^6)$, and consider the element $\varphi(b)\in \tilde{H}^5(C)$. Passing freely between reduced and relative cohomology, by the Cartan formula we have
 \begin{align*}
 Sq^2(\varphi(b)) & = Sq^2(\pi_1^*(b)\cup U) \\
                          & = Sq^2(\pi_1^*(b))\cup U + Sq^1(\pi_1^*(b))\cup Sq^1(U) + \pi_1^*(b)\cup Sq^2(U)\\
                          & = \pi_1^*(Sq^2(b))\cup U\\
                          & = \pi_1^*(b^2)\cup U\\
                          & = \varphi(b^2).
\end{align*}
In particular, $Sq^2:H^5(C)\to H^7(C)$ is an isomorphism.

We now wish to apply \cite[Prop 4.2.2]{Harper} and the cofibre sequence
\[
\xymatrix{
E \ar[r]^f & E(1) \ar[r]^j & C \ar[r]^\tau& \Sigma E \ar[r]^-{\Sigma f} & \Sigma E(1)
}
\]
to compute the operation $\Phi:H^5(E(1)) \to H^8(E(1))$. The relevant diagram (compare \cite[p. 93]{Harper}) is
\[
\xymatrix{
  & H^6(\Sigma E(1))\times H^7(\Sigma E(1)) \ar@{=}[r] \ar[d]_{\Sigma f^*} & H^6(\Sigma E(1))\times H^7(\Sigma E(1)) \ar[d]_B \\
   & H^6(\Sigma E)\times H^7(\Sigma E) \ar[r]^{\begin{pmatrix}Sq^3 \\Sq^2\end{pmatrix}} \ar[d]_{\tau^*} & H^9(\Sigma E)\\
 H^5(C) \ar[r]^-{\begin{pmatrix} Sq^1 & Sq^2 \end{pmatrix}} \ar@{=}[d] &  H^6(C)\times H^7(C) \ar[d]_{j^*} & \\
 H^5(C) \ar[r]^-{A} & H^6(E(1))\times H^7(E(1))
 }
 \]
 where the operations $A$ and $B$ are defined by commutativity of the squares, and the diagram gives a well-defined operation
 \[
 \Delta: \ker A \to \coker B, \qquad \Delta=\begin{pmatrix}Sq^3 \\Sq^2\end{pmatrix}\circ (\tau^*)^{-1} \circ (Sq^1,Sq^2).
 \]
 In our case, the operation $A$ is trivial since $H^6(E(1))=H^7(E(1))=0$, and the operation $B$ is trivial since $H^6(\Sigma E)=H^5(\C P^5)=0$ and $H^7(\Sigma E(1))=H^6(E(1))=0$. Thus we get a well-defined operation $\Delta: H^5(C) \to H^9(\Sigma E)$. Then, working mod $2$, \cite[Prop 4.2.2]{Harper} says that the diagram
 \[
 \xymatrix{
 H^5(C) \ar[r]^{\Delta} \ar[d]_{j^*} & H^9(\Sigma E)\\
 H^5(E(1)) \ar[r]^-{\Phi} & H^8(E(1))\cong H^9(\Sigma E(1)) \ar[u]_{\Sigma f^*}
 }
 \]
 commutes.  Finally, we note that 
 \[
 Sq^2: H^5(C) \to H^7(C),
 \]
 \[
 \tau^*: H^7(\Sigma E)\to H^7(C),
 \]
 \[
 Sq^2: H^7(\Sigma E)\to H^9(\Sigma E)
 \]
 are all isomorphisms. The first follows from the Thom isomorphism as above, the second since $H^7(\Sigma E(1))=H^7(E(1))=0$, and the third since in $H^*(\C P^5)\cong \F_2[b]/(b^6)$ we have 
 \[
 Sq^2(b^3)=3 Sq^2(b)\cup b^2=b^4
 \]
 by the Cartan formula. Hence $\Delta$ is nonzero, which implies that $\Phi$ is nonzero. This proves the claim and the lemma.
 \end{proof}
 
 Theorem \ref{thm:twistor} is immediate on combining Lemmas \ref{lem:Mwgttwistor} and \ref{lem:Swgttwistor}. \qed

\begin{rem}
The primary obstruction to sectioning $q(1)$ is the Euler class in $H^6(\Ha P^2;\pi_5(S^5))=0$. The next obstruction would also live in a zero group $H^7(\Ha P^2;\pi_6(S^5))$. So there exist sections of $q(1)$ on the $7$-skeleton of $\Ha P^2$, but each of these are obstructed from extending to all of $\Ha P^2$ by the non-trivial class in $H^8(\Ha P^2;\pi_7(S^5))=H^8(\Ha P^2;\Z_2)\cong\Z_2$. Incidentally, this class is not in the kernel of $q(1)^*$, which illustrates the non-naturality of higher obstructions.
\end{rem}

\begin{rem}
Our proof is modelled closely on the proof that $\eta\circ \eta\in \pi_{n+2}(S^n)$ is detected by secondary operations in the cofibre $C_{\eta\circ\eta}=S^n\cup_{\eta\circ\eta} e^{n+3}$ given on pp. 95---97 of \cite{Harper}. In fact, our proof shows that $E(1)$ has a cell structure whose $8$-skeleton is $(S^4\vee S^5)\cup_f e^8$, where the map $f$ projects to the Hopf element $\sigma\in \pi_7(S^4)$ on the $S^4$ summand and $\eta\circ\eta\in \pi_7(S^5)$ on the $S^5$ summand.
\end{rem}

\section{Proof of Theorem \ref{thm:2cell}}

Let $\alpha\in \pi_7(S^3)$ be the homotopy class of the composition
\[
\xymatrix{
S^7 \ar[r]^{\eta} & S^6 \ar[r]^{\omega} & S^3,
}
\]
where $\eta$ denotes a suspension of the Hopf map $S^3\to S^2$ and $\omega$ generates $\pi_6(S^3)\cong \Z_{12}$. Define $X:=C_\alpha=S^3\cup_\alpha e^8$ to be the homotopy cofibre of $\alpha$.

 By \cite[Ex 6.11, Cor 6.23, Prop 6.35]{CLOT}, the Berstein-Hilton Hopf invariant $H(\alpha)\in \pi_7(\Omega S^3 \ast \Omega S^3)$ of $\alpha$ is represented by the composition
\[
\xymatrix{
S^7 \ar[r]^{\eta} & S^6 \ar[r]^{\eta} & S^5.
}
\]
Since $\eta\circ\eta\in\pi_7(S^5)$ is nonzero, $\cat(X)=2$. Our objective here is to show that this lower bound can be recovered using secondary cohomology operations, but not using primary operations. In particular, we will show that 
\[
\Swgt(X;\F_2)=2>1=\Mwgt(X;\F_2)=\wgt(X;\F_2),
\]
thus proving Theorem \ref{thm:2cell}. The proof will follow a series of preliminary results. 

Consider the diagram below, whose rows are cofibre sequences:
\begin{equation}\label{cofibdiag}
\xymatrix{
S^7 \ar[r]^\eta \ar@{=}[d] & S^6 \ar[r] \ar[d]^{\omega} & C_{\eta}\ar[d]^\psi \\
S^7 \ar[r]^\alpha & S^3 \ar[r] & C_{\alpha}=X
}
\end{equation}
The map $\psi: C_\eta \to C_\alpha$ is an induced map of cofibres. Applying the functor $G_1\simeq\Sigma\Omega$ gives a cofibre sequence
\begin{equation}\label{cofib}
\xymatrix{
G_1(C_\eta) \ar[r]^{G_1(\psi)} & G_1(C_{\alpha}) \ar[r]^-j & C:=C_{G_1(\psi)} \ar[r]^\tau & \Sigma G_1(C_\eta) \ar[r]^-{\Sigma G_1(\psi)} & \Sigma G_1(C_\alpha).
}
\end{equation}

For the rest of this section $H^*(-):= H^*(-;\F_2)$. We aim to compute secondary operations in $H^*(G_1(C_\alpha))$ by applying compatibility with exact sequences \cite[Prop 4.2.2]{Harper} to this cofibre sequence. First we collect some information about the cohomology of the spaces involved. 

\begin{lem}\label{cohomY}
The (reduced) cohomology groups of $G_1(C_\eta)$ are given for $i\le15$ by 
\[
 H^i(G_1(C_\eta))=\begin{cases} \F_2 & \mbox{if }i=6,8,11,13,15,\ldots \\ 0 & \mbox{otherwise.} \end{cases}
 \]
 Furthermore, $Sq^2:H^6(G_1(C_\eta))\to H^8(G_1(C_\eta))$ is an isomorphism.
 \end{lem}
 
 \begin{proof}
 Note that $C_\eta\simeq \Sigma^4\C P^2$. The differential in the Adams--Hilton model of $C_\eta$ is zero for degree reasons, and therefore $H_*(\Omega C_\eta)\cong \mathbb{T}(a_5,b_7)$. The homology, and thus also the cohomology, of $\Omega C_\eta$ is therefore concentrated in degrees $5,7,10,12,14,\ldots$ which gives the claim about the cohomology groups of $G_1(C_\eta)\simeq\Sigma\Omega C_\eta$. Note that $Sq^2 :H^6(C_\eta)\to H^8(C_\eta)$ is an isomorphism, and since the cohomology loop suspension is an isomorphism in these dimensions, so are $Sq^2: H^5(\Omega C_\eta)\to H^7(\Omega C_\eta)$ and $Sq^2: H^6(G_1(C_\eta))\to H^8(G_1(C_\eta))$. 
 \end{proof}
 
 \begin{lem}\label{cohomX}
 The (reduced) cohomology groups of $G_1(C_\alpha)$ are given for $i\le 11$ by
 \[
 H^i(G_1(C_\alpha))=\begin{cases} \F_2 & \mbox{if }i=3,5,7,8,9,10,11,\ldots \\ 0 & \mbox{otherwise.} \end{cases}
 \]
 Furthermore, $Sq^2:H^5(G_1(C_\alpha))\to H^7(G_1(C_\alpha))$ and $Sq^3:H^5(G_1(C_\alpha))\to H^8(G_1(C_\alpha))$ are both trivial.
 \end{lem}
 
 \begin{proof}
 The differential in the Adams--Hilton model of $C_\alpha$ is zero since the adjoint of $\alpha$ factors as
 \[
 \xymatrix{
 S^6 \ar[r] & S^5 \ar[r] & \Omega S^3.
 }
 \]
 Therefore $H_*(\Omega C_\alpha)\cong \mathbb{T}(a_2,b_7)$. The homology, and thus also the cohomology, of $\Omega C_\alpha$ is therefore concentrated in degrees $2,4,6,7,8,9,10,\ldots$ which gives the claim about the cohomology groups of $G_1(C_\alpha)\simeq\Sigma\Omega C_\alpha$. To verify the claim about $Sq^2$, it suffices to show that $Sq^2:H^4(\Omega C_\alpha)\to H^6(\Omega C_\alpha)$ is zero, or equivalently its linear dual $Sq^2_*:H_6(\Omega C_\alpha)\to H_4(\Omega C_\alpha)$ is zero. This follows from the description of the Pontryagin algebra $H_*(\Omega C_\alpha)$ given above together with the Cartan formula for Pontryagin products, which gives $Sq^2_*(a^3)=0$. 
 
 To prove the claim about $Sq^3$, inspect the diagram
 \[
 \xymatrix{
 H^5(G_1(C_\alpha)) \ar[r]^-{G_1(\psi)^*} \ar[d]_{Sq^3} & H^5(G_1(C_\eta)) \ar[d]^{Sq^3=0} \\
 H^8(G_1(C_\alpha)) \ar[r]^-{G_1(\psi)^*} & H^8(G_1(C_\eta))
 }
 \]
 and note that the bottom horizontal arrow is an isomorphism as follows from applying the cohomology suspension to the isomorphism $\psi^*: H^8(C_\alpha)\to H^8(C_\eta)$. 
  \end{proof}
 
  The following must be well known, but we've not yet found a reference. Recall that a map is called an \emph{$m$-equivalence} if it induces an isomorphism on homotopy groups in dimensions less than $m$, and an epimorphism in dimension $m$.
 
 \begin{lem}\label{Ganea}
 Let $f:A\to X$ be a map with cofiber $C_f$. Suppose that $A$ is $(n-1)$-connected and $f$ is an $m$-equivalence, and that $m,n\ge1$. Then the natural map $C_{\Omega f}\to \Omega C_f$ is an $(m+n-1)$-equivalence. Hence the natural map $C_{G_1(f)}\to G_1(C_f)$ is an $(m+n)$-equivalence.
 \end{lem}
 
 \begin{proof}
 Let $\rho:X\to C_f$ be the projection to the cofiber. Then we may construct a commuting diagram whose rows are fibration sequences:
 \[
 \xymatrix{
 \Omega F_\rho \ar[r] & \Omega X \ar[r] & \Omega C_f \ar[r] & F_\rho \ar[r] & X \ar[r]^-{\rho} & C_f \\
 \Omega A \ar[r]^-{\Omega f} \ar[u] & \Omega X \ar@{=}[u] \ar[r] \ar[dr] & F_f \ar[r] \ar[u]^{\varphi} & A \ar[r]^{f} \ar[u] & X \ar@{=}[u] & \\
  & & C_{\Omega f} \ar[u]^{\xi} & & & 
  }
  \]
According to Ganea \cite{Ganea65}, the homotopy fibre of the natural map $\xi: C_{\Omega f}\to F_f$ is $\Omega A \ast \Omega F_f$, which is $(m+n-2)$-connected. Hence $\xi$ is an $(m+n-1)$-equivalence. Also according to Ganea \cite[Theorem 3.1]{Ganea68}, the natural map $\varphi: F_f \to \Omega C_f$ is an $(m+n-1)$-equivalence. Then $\varphi\circ\xi$ is an $(m+n-1)$-equivalence. The claim follows.
 \end{proof}

 \begin{lem}\label{cohomC}
 The cohomology groups of $C:=C_{G_1(\psi)}$ satisfy $H^5(C)\cong \F_2$ and $H^7(C)\cong \F_2\oplus\F_2$. Furthermore, $Sq^2:H^5(C)\to H^7(C)$ is non-trivial.
 \end{lem}
 
 \begin{proof}
 The cofiber of the map $\psi:C_\eta\to C_\alpha$ is $C_{\omega}$, as illustrated by the following extension of diagram (\ref{cofibdiag}) whose rows and columns are cofibre sequences:
 \[
\xymatrix{
S^7 \ar[r]^\eta \ar@{=}[d] & S^6 \ar[r] \ar[d]^{\omega} & C_{\eta} \ar[d]^\psi \\
S^7 \ar[r]^\alpha \ar[d]& S^3 \ar[r] \ar[d] & C_{\alpha} \ar[d] \\
\ast \ar[r] & C_{\omega} \ar@{=}[r] & C_{\omega}
}
\]
 Now $C_\eta$ is $(6-1)$-connected and $\psi$ is a $2$-equivalence, so Lemma \ref{Ganea} gives an $8$-equivalence $C\to G_1(C_{\omega})$. One can now verify the claim about the cohomology groups of $C$, either by examining the Adams--Hilton model for $C_\omega$ as in Lemmas \ref{cohomX} and \ref{cohomY}, or by examining the relevant portions of the long exact cohomology sequence of (\ref{cofib}):
 \[
 \xymatrix{
 0 \ar[r] & H^5(C) \ar[r] & H^5(G_1(C_\alpha)) \ar[r] & 0,
 }
 \]
 \[
 \xymatrix{
 0\ar[r] & H^6(G_1(C_\eta)) \ar[r] & H^7(C) \ar[r] & H^7(G_1(C_\alpha)) \ar[r] & 0
 }
 \]
 
 To prove the claim about $Sq^2$ it suffices to show that $Sq^2:H^5(G_1(C_{\omega}))\to H^7(G_1(C_{\omega}))$ is non-trivial, or equivalently that $Sq^2: H^4(\Omega C_{\omega})\to H^6(\Omega C_{\omega})$ is non-trivial. We next observe that $C_\omega=S^3\cup_\omega e^7$ is the $9$-skeleton of the Lie group $Sp(2)\cong Spin(5)\simeq (S^3\cup_{\omega} e^7)\cup e^{10}$, giving an $8$-equivalence $\Omega C_\omega\to \Omega Spin(5)$. We are therefore reduced to showing that $H^4(\Omega Spin(5))\to H^6(\Omega Spin(5))$ is non-trivial. This follows from a result of Kono--Kozima \cite[Prop 3.4]{KonoKozima}. 
 
 [In slightly more detail: letting $Spin(5)_{\langle 3\rangle}$ denote the $3$-connected cover of $Spin(5)$, Kono--Kozima show that $Sq^2: H^5(Spin(5)_{\langle 3\rangle})\to H^7(Spin(5)_{\langle 3\rangle})$ is non-trivial. Since the cohomology suspension $\sigma: H^7(Spin(5)_{\langle 3\rangle})\to H^6(\Omega Spin(5)_{\langle 3\rangle})$ is injective, it follows that $Sq^2: H^4(\Omega Spin(5)_{\langle 3\rangle})\to H^6(\Omega Spin(5)_{\langle 3\rangle})$ is non-trivial. Now since $H^*(\Omega Spin(5)_{\langle 3\rangle})\to H^*(\Omega Spin(5))$ is epimorphic in even degrees, it follows that $Sq^2: H^4(\Omega Spin(5))\to H^6(\Omega Spin(5))$ is non-trivial.]
 \end{proof}
 
 \begin{rem}
 Lucile Vandembroucq has pointed out the following alternative proof that $Sq^2:H^5(G_1(C_{\omega}))\to H^7(G_1(C_{\omega}))$ is non-trivial, based on results of Gilbert \cite{Gilbert}. Consider the cofiber sequence
 \[
 \xymatrix{
 \Sigma S^5 = S^6 \ar[r]^-{\omega} & S^3 \ar[r] & C_\omega.
 }
 \]
 Let $\omega^\perp: S^5\to \Omega S^3$ be the adjoint of $\omega$, and $\beta:=\Sigma\omega^\perp: S^6\to \Sigma\Omega S^3$ its suspension. According to \cite[Proposition 4.3]{Gilbert}, there is an $8$-equivalence $\Sigma\Omega S^3 \cup_\beta e^7 \to G_1(C_\omega)$. For dimensional reasons, after the identification $\Sigma \Omega S^3\simeq S^3 \vee S^5 \vee S^7\vee\cdots$ the map $\beta$ is completely determined by its projections onto the bottom two spheres $S^3$ and $S^5$, which can be indentified with $\omega$ and $H_2(\omega)=\eta$ \cite{CLOT, Gilbert}. The cofibre of the map $S^3\to \Sigma\Omega S^3 \cup_\beta e^7$ which includes the bottom cell therefore has the homotopy type of $S^5\cup_\eta e^7$ wedge spheres, and the map to this cofibre induces isomorphisms on cohomology in degrees $5$ and $7$.
 \end{rem}
 
 \begin{prop}
 The secondary operation $\Phi:H^5(G_1(C_\alpha))\to H^8(G_1(C_\alpha))$ based on $Sq^3Sq^1 + Sq^2 Sq^2\equiv 0$ is non-trivial.
 \end{prop}
 
 \begin{proof}
 We apply Proposition \ref{prop:compat} to the cofiber sequence (\ref{cofib}), the relevant diagram being
\[
\xymatrix{
  & H^6(\Sigma G_1(C_\alpha))\times H^7(\Sigma G_1(C_\alpha)) \ar@{=}[r] \ar[d]_{\Sigma G_1(\psi)^*} & H^6(\Sigma G_1(C_\alpha))\times H^7(\Sigma G_1(C_\alpha)) \ar[d]_B \\
   & H^6(\Sigma G_1(C_\eta))\times H^7(\Sigma G_1(C_\eta)) \ar[r]^-{\begin{pmatrix}Sq^3 \\Sq^2\end{pmatrix}} \ar[d]_{\tau^*} & H^9(\Sigma G_1(C_\eta))\\
 H^5(C) \ar[r]^-{\begin{pmatrix} Sq^1 & Sq^2 \end{pmatrix}} \ar@{=}[d] &  H^6(C)\times H^7(C) \ar[d]_{j^*} & \\
 H^5(C) \ar[r]^-{A} & H^6(G_1(C_\alpha))\times H^7(G_1(C_\alpha))
 }
 \]
 where the operations $A$ and $B$ are defined by commutativity of the squares, and the diagram gives a well-defined operation
 \[
 \Delta: \ker A \to \coker B, \qquad \Delta=\begin{pmatrix}Sq^3 \\Sq^2\end{pmatrix}\circ (\tau^*)^{-1} \circ (Sq^1,Sq^2).
 \]
 The operation $A$ is trivial, since $H^6(G_1(C_\alpha))=0$ and $j^*Sq^2=Sq^2 j^*:H^5(C)\to H^7(G_1(C_\alpha))$ is trivial by Lemma \ref{cohomX}. The operation $B$ is trivial since $H^6(\Sigma G_1(C_\eta))$ and $H^7(\Sigma G_1(C_\alpha))$ are both trivial by Lemmas \ref{cohomY} and \ref{cohomX}. Thus we get a well-defined operation $\Delta: H^5(C) \to H^9(\Sigma G_1(C_\eta))$. Then Proposition \ref{prop:compat} says that the diagram
 \[
 \xymatrix{
 H^5(C) \ar[r]^{\Delta} \ar[d]_{j^*} & H^9(\Sigma G_1(C_\eta))\\
 H^5(G_1(C_\alpha)) \ar[r]^-{\Phi} & H^8(G_1(C_\alpha))\cong H^9(\Sigma G_1(C_\alpha)) \ar[u]_{\Sigma G_1(\psi)^*}
 }
 \]
 commutes.
 
 Finally we note that $\Delta$ is non-trivial: If $c\in H^5(C)$ is the generator then $Sq^2(c)\in H^7(C)$ is a nonzero element of $\ker j^*=\mathrm{im}\, \tau^*$ by Lemma \ref{cohomC}, and $Sq^2: H^7(\Sigma G_1(C_\eta))\to H^9(\Sigma G_1(C_\eta))$ is an isomorphism by Lemma \ref{cohomY}. 
 \end{proof} 
 

\begin{proof}[Proof of Theorem \ref{thm:2cell}]
Examining the Adams--Hilton model for $\Omega X$ as in the proof of Lemma \ref{cohomX} shows that $H^*(X)\to H^*(G_1(X))$ is injective, therefore $\operatorname{wgt}(X;\F_2)=1$. To see that $\operatorname{Mwgt}(X;\F_2)\leq 1$ we must produce a retraction $H^*(G_1(X))\to H^*(X)$ as $\mathcal{A}_2$-modules. This is easy enough, since all Steenrod squares in $H^*(X)$ are trivial and there are no non-trivial Steenrod operations in $H^*(G_1(X))$ taking values in the image of $H^*(X)$. 

Finally, to see that $\operatorname{Swgt}(X;\F_2)>1$ we simply observe that there cannot exist a retraction $H^*(G_1(X))\to H^*(X)$ commuting with the secondary operation $\Phi$, since $H^5(X)=0$. Since the two-cell complex $X$ must have $\cat(X)\le 2$, this completes the proof. 
\end{proof}

\begin{rem}
Similar arguments in $\F_p$-cohomology yield that $\Mwgt(X;\F_p)=1$.
\end{rem}

\begin{rem}
Note that higher Hopf invariants were not used in the above proof, other than to discover the example. We regard the above example as a ``proof-of-concept", suggesting the secondary module weight may be a useful estimate of category in more complicated examples.
\end{rem}


\begin{thebibliography}{99}

\bibitem{BH} I.\ Berstein, P.\ Hilton, \emph{Category and generalized Hopf invariants}, Illinois J. Math. {\bf 4} (1960), 437--451.

\bibitem{C-VKV} J.\ Carrasquel-Vera, T.\ Kahl, L.\ Vandembroucq, \emph{Rational approximations of sectional category and Poincar\'e duality}, Proc. Amer. Math. Soc. {\bf 144} (2016), no. 2, 909--915.

\bibitem{Chakra} P.\ Chakraborty, \emph{Sectional category of a map}, \url{https://mathoverflow.net/q/495056/}{mathoverflow.net/495056}

\bibitem{Choi09} Y.\ Choi, \emph{On the category weight of $\mathrm{Spin}(n)$}, Topology Appl. {\bf 156} (2009), no. 14, 2370--2375.

\bibitem{Choi15} Y.\ Choi, \emph{The module category weight of compact exceptional Lie groups}, Topol. Methods Nonlinear Anal. {\bf 45} (2015), no. 1, 157--168.

\bibitem{CFW} D.\ C.\ Cohen, M.\ Farber, S.\ Weinberger, \emph{Topology of parametrized motion planning algorithms}, SIAM J. Appl. Algebra Geom. {\bf 5} (2021), no. 2, 229--249.

\bibitem{CLOT} O.\ Cornea, G.\ Lupton, J.\ Oprea, D.\ Tanr\' e, \emph{Lusternik-Schnirelmann category}, Mathematical Surveys and Monographs, 103, AMS, Providence, RI, 2003.

\bibitem{FH}  E.\ Fadell, S.\ Husseini, \emph{Category weight and Steenrod operations}, Papers in honor of Jos\'e Adem (Spanish). Bol. Soc. Mat. Mexicana (2) {\bf 37} (1992), no. 1-2, 151--161.

\bibitem{Far03} M.\ Farber, \emph{Topological complexity of motion planning}, Discrete Comput. Geom. {\bf 29} (2003), no. 2, 211--221.

\bibitem{Far04} M.\ Farber, \emph{Instabilities of robot motion}, Topology Appl. {\bf 140} (2004), no. 2-3, 245--266.

\bibitem{FG07} M.\ Farber, M.\ Grant, \emph{Symmetric motion planning}, Topology and robotics, 85--104, Contemp. Math., {\bf 438}, Amer. Math. Soc., Providence, RI, 2007.

\bibitem{FG08} M.\ Farber, M.\ Grant, \emph{Robot motion planning, weights of cohomology classes, and cohomology operations}, Proc. Amer. Math. Soc. {\bf 136} (2008), no. 9, 3339--3349.
 
\bibitem{Ganea65} T.\ Ganea, {\em A generalization of the homology and homotopy suspension}, Comment. Math. Helv. {\em 39} (1965), 295--322.

\bibitem{Ganea68} T.\ Ganea, {\em On the homotopy suspension}, Comment. Math. Helv. {\bf 43} (1968), 225--234.

\bibitem{Gilbert} W.\ J.\ Gilbert, {\em Some examples for weak category and conilpotency}, Illinois J. Math. {\bf 12} (1968), 421--432.

\bibitem{GGV} J.\ Gonz\'{a}lez, M.\ Grant, L.\ Vandembroucq, \emph{Hopf invariants for sectional category with applications to topological robotics}, Q. J. Math. {\bf 70} (2019), no. 4, 1209--1252.

\bibitem{Harper} J.\ R.\ Harper, \emph{Secondary cohomology operations}, Graduate Studies in Mathematics, {\bf 49}. American Mathematical Society, Providence, RI, 2002.

\bibitem{IwaseKono} N.\ Iwase, A.\ Kono, \emph{Lusternik-Schnirelmann category of $\operatorname{Spin}(9)$}, Trans.\ Amer.\ Math.\ Soc. {\bf 359} (2007), no. 4, 1517--1526.

\bibitem{James} I.\ M.\ James, \emph{On category, in the sense of Lusternik-Schnirelmann}, Topology {\bf 17} (1978), no. 4, 331--348.

\bibitem{KonoKozima} A.\ Kono, K.\ Kozima, {\em The adjoint action of a Lie group on the space of loops}, J. Math. Soc. Japan {\bf 45} (1993), no. 3, 495--510.

\bibitem{Rudyak99} Y.\ B.\ Rudyak, \emph{On category weight and its applications}, Topology {\bf 38} (1999), no. 1, 37--55.
 
\bibitem{Rudyak} Y.\ B.\ Rudyak, \emph{On higher analogs of topological complexity}, Topology Appl. {\bf 157} (2010), no. 5, 916--920.

\bibitem{Schwarz} A.\ S.\ Schwarz,  \emph{The genus of a fiber space},
 Amer.\ Math.\ Soc.\ Transl.  {\bf 55} (1966), no. 2, 49--140.

\bibitem{Stanley} D.\ Stanley, \emph{The sectional category of spherical fibrations}, Proc.\ Amer.\ Math.\ Soc. {\bf 128} (2000), no. 10, 3137--3143.

\bibitem{Strom} J.\ A.\ Strom, \emph{Category weight and essential category weight}\, Thesis (Ph.D.)–The University of Wisconsin - Madison. 1997. 139 pp.
 
\bibitem{Thom} R.\ Thom, \emph{Espaces fibr\'{e}s en sph\`{e}res et carr\'{e}s de Steenrod}, (French) Ann.\ Sci. \'{E}cole Norm.\ Sup. (3) {\bf 69} (1952), 109--182.

\bibitem{Toomer} G.\ H.\ Toomer, \emph{Lusternik-Schnirelmann category and the Moore spectral sequence}, Math.\ Z. {\bf 138} (1974), 123--143.

\end{thebibliography}
\end{document}